%% file: PRH1C.tex
\documentclass[10pt]{amsart}

\usepackage{amsmath,amssymb,amsthm}

\usepackage[utf8]{inputenc}

\usepackage[english]{babel}

\numberwithin{equation}{section}

\newcommand{\R}{\mathbb{R}}

\newcommand{\N}{\mathbb{N}}
\newcommand{\abs}[1]{\mathord{\left|#1\right|}}
\newcommand{\zabs}[1]{\mathord{\left|\;#1\right|}}
\newcommand{\eps}{\varepsilon}
\newcommand{\half}{\frac{1}{2}}
\newcommand{\zsum}{\sideset{}{^*}\sum}
\DeclareMathOperator{\li}{li}

\newtheorem{prop}{Proposition}
\newtheorem{thm}{Theorem}
\newtheorem{lemma}{Lemma}

\newcommand{\lce}{\lambda_{c,\eps}}
\newcommand{\nce}{\eta_{c,\eps}}
\newcommand{\CE}{\mathcal{E}}

\title[Estimates for $\pi(x)$ et al]{Estimating $\pi(x)$ and related functions under partial RH assumptions}

\author{Jan Büthe}
\address{Mathematisches Institut, Endenicher Allee 60, 53115 Bonn}
\email{jbuethe@math.uni-bonn.de}

\date{\today}

\begin{document}
\begin{abstract}
We give a direct interpretation of the validity of the Riemann hypothesis for all zeros with $\Im(\rho)\in(0,T]$ in terms of the prime-counting function $\pi(x)$, by proving that Schoenfeld's explicit estimates for $\pi(x)$ and the Chebyshov functions hold as long as $4.92\sqrt{x/\log(x)} \leq T$.

We also improve some of the existing bounds of Chebyshov type for the function $\psi(x)$.
\end{abstract}

\maketitle

\section{introduction}
The Riemann hypothesis has been subject to numerous numerical verifications, which typically lead to statements of the form \textit{the first $n$ complex zeros of the Riemann zeta function are  simple and lie on the critical line $\Re(s)=1/2$}; see e.g. \cite{Brent79}.

Whilst such results are used as an ingredient in many estimates for functions of prime numbers, it is the purpose of this paper to give a direct interpretation in terms of the prime-counting function $\pi(x)$. This is done by proving the well-known Schoenfeld bound
\begin{equation*}
\abs{\pi(x) - \li(x)} \leq \frac{\sqrt{x}}{8\pi}\log(x) \quad\quad\text{for $x>2657$},
\end{equation*}
which is implied by the Riemann hypothesis \cite{Schoenfeld76}, holds for $4.92\sqrt{x/\log(x)} \leq T$ conditional on the Riemann hypothesis being valid for $0<\Im(\rho)\leq T$. We also prove equivalent statements for the Riemann prime-counting function and the Chebyshov functions.

These results also have practical relevance, since calculating the zeros up to height $T$ with fast methods like the Odlyzko-Sch\"onhage algorithm has expected run time $O(T^{1+\eps})$ \cite{OS88}. Therefore, one obtains strong bounds for $\pi(x)$ for $x\leq x_1$ in expected run time $O(x_1^{1/2+\eps})$ if the Riemann hypothesis holds up to the according height.

Apart from this, we also improve part of the bounds for $\psi(x)$ given in \cite{FK15}.

\input{modified-PsiX.tex}

\input{explicit-formula.tex}

\input{zsum.tex}

\input{psum.tex}

\input{chebb.tex}

\input{ppnt.tex}

\section*{Acknowledgment}
The author wishes to thank Jens Franke for pointing out the Sturm Monotony Principle, which was essential for the proof
of Lemma \ref{l:BesselQuot-Monotonie}. He also wishes to thank the referee for his or her comments. Furthermore, he wishes to thank Allysa Lumley for pointing out a mistaken sign in Lemma \ref{l:riemann-sums}.

\bibliographystyle{amsalpha}
\bibliography{/home/jbuethe/TeX/inputs/jankabib.bib}

\end{document}

%% file: modified-PsiX.tex
\section{A modified Chebyshov function}
For $A\subset X$ let
\[
 \chi^*_A(x) = \begin{cases}
              1		&x\in A\setminus \partial A\\
	      1/2	&x\in\partial A\\
	      0		&x \in X\setminus \overline A.
             \end{cases}
\]
denote the normalized characteristic function. We intend to construct a continuous approximation to the (normalized) Chebyshov function
\[
\psi(x) = \sum_{p^m} \chi^*_{[0,x]}( \log p),
\]
for which we will prove an explicit formula similar to the von Mangoldt explicit formula
\begin{equation}\label{e:mangoldt-explicit-formula}
\psi(x) = x - \zsum_\rho \frac{x^\rho}{\rho} - \log(2\pi) - \frac 12 \log(1-x^{-2}),
\end{equation}
where the sum is taken over all non-trivial zeros (according to their multiplicity) of the Riemann zeta function and  the $*$ indicates that the sum is computed as
\[
\lim_{T\to\infty} \sum_{\abs{\Im(\rho)}<T} \frac{x^\rho}{\rho}
\]
\cite{Mangoldt95}.

To this end, we use the Fourier transform of the Logan function 
\begin{equation*}
\ell_{c,\eps}(\xi) = \frac{c}{\sinh c}\frac{\sin(\sqrt{(\xi\eps)^2-c^2})}{\sqrt{(\xi\eps)^2-c^2}},
\end{equation*}
a sharp cuttoff filter kernel \cite{Logan88}, which will allow us to flexibly control the truncation point and the size of the remainder term of the sum over zeros. The Fourier transform is given by
\begin{equation}\label{e:eta-def}
\eta_{c,\eps}(t) = \frac{1}{2\pi}\int_\R e^{-it\xi} \ell_{c,\eps}(\xi)\, d\xi =\chi^*_{[-\eps,\eps]}(t) \frac{c}{2\eps\sinh c} I_0\bigl(c\sqrt{1-(t/\eps)^2}\bigr)
\end{equation}
where $I_0(t) = \sum_{n=0}^\infty (t/2)^{2n}/(n!)^2$ denotes the 0-th modified Bessel function of the first kind \cite{FKBJ}.

Now let $\lambda_{c,\eps} = \ell_{c,\eps}(i/2)$ and let
\begin{equation*}
\varphi_{x,c,\eps} = \frac{1}{\lambda_{c,\eps}} \bigl(\chi_{[0,\log x]} \exp(\cdot/2)\bigl)\ast \eta_{c,\eps},
\end{equation*}
where, as usual,
\[
f\ast g(x) := \int_\R f(y)g(x-y)\, dy
\]
denotes the convolution of two functions. Then we define the modified Chebyshov function by
\begin{equation*}
\psi_{c,\eps}(x) = \sum_{p^m} \frac{\log p}{p^{m/2}} \varphi_{x,c,\eps}(m\log p).
\end{equation*}

\begin{prop}\label{p:psi-psice}
Let $\eps < 1/10$ and let
\begin{multline}\label{e:M-def}
M_{x,c,\eps} (t) = \frac{\log t}{\lambda_{c,\eps}} \Bigl[\chi^*_{[x,\exp(\eps) x]}(t) \int_{-\eps}^{\log(t/x)}\eta_{c,\eps}(\tau) e^{-\tau/2}\, d\tau \\
- \chi^*_{[\exp(-\eps)x,x]}(t) \int_{\log(t/x)}^\eps \eta_{c,\eps}(\tau) e^{-\tau/2}\, d\tau
\Bigr].
\end{multline}
Then we have
\begin{equation}\label{e:psi-psice-equal}
\psi(x) = \psi_{c,\eps}(x) - \sum_{e^{-\eps}x < p^m < e^{\eps}x} \frac{1}{m}M_{x,c,\eps}(p^m).
\end{equation}
Moreover, we have
\begin{align}
\psi(e^{-\alpha\eps} x) &\leq \psi_{c,\eps}(x) - \sum_{e^{-\eps}x< p^m \leq e^{-\alpha\eps} x } \frac{1}{m} M_{x,c,\eps}(p^{m})\label{e:psi-psice-upper} \\
\intertext{and}
\psi(e^{\alpha\eps} x) &\geq \psi_{c,\eps}(x) - \sum_{e^{\alpha\eps}x \leq p^m <e^{\eps} x } \frac{1}{m} M_{x,c,\eps}(p^{m})\label{e:psi-psice-lower}
\end{align}
for every $\alpha >0$.
\end{prop}
\begin{proof}
The identity \eqref{e:psi-psice-equal} follows directly from
\begin{equation*}
 \exp(\cdot/2)\ast \eta_{c,\eps} (t) = \lambda_{c,\eps} \exp(t/2). 
\end{equation*}
and from $\eta_{c,\eps}$ being compactly supported on $[-\eps,\eps]$. The inequalities \eqref{e:psi-psice-upper} and \eqref{e:psi-psice-lower} then follow from \eqref{e:psi-psice-equal}, since \eqref{e:eta-def} implies $\eta_{c,\eps}(t)\geq 0$. 
\end{proof}

%% file: explicit-formula.tex
\section{The explicit formula}
The modified Chebyshov function satisfies an explicit formula similar to \eqref{e:mangoldt-explicit-formula}, of which we prove an approximate version.

\begin{prop}\label{p:exp-form}
Let $0 < \eps < 1/10$ and let  $\log(x) > 2/\abs{\log\eps}$. We define 
\[C_1 = -\gamma/2 - 1 - \log(\pi)/2\]
and
\[
a_{c,\eps}(\rho) = \frac{1}{\lambda_{c,\eps}} \ell_{c,\eps}\Bigl(\frac{\rho}{i} - \frac{1}{2i}\Bigr).
\]
Then we have
\begin{equation}\label{e:exp-form}
\psi_{c,\eps}(x) = x - \sum_\rho a_{c,\eps}(\rho) \frac{x^\rho-1}{\rho} + C_1 - \frac{1}{2}\log(1-x^{-2}) + \Theta(8 \eps\abs{\log\eps}).
\end{equation}
\end{prop}
\begin{proof}
Let
\[f_x(t) = \chi^*_{[0,\log x]}(t) \exp(t/2)\]
so that we have $\varphi_{x,c,\eps} = \lambda_{c,\eps}^{-1} f_x\ast \eta_{c,\eps}.$ The assertion of the proposition follows by applying the Weil-Barner explicit formula \cite{Barner81}
\[
w_s(\hat f) = w_f(f) + w_\infty(f),
\]
where
\begin{align*}
w_s(\hat f) &= \zsum_{\rho} \hat f(i/2-i\rho) - \hat f(i/2) - \hat f(i/2), \\
w_f(f) &= -\sum_{p} \sum_{m=0}^\infty \frac{\log p}{p^{m/2}}\bigl(f(m\log p) + f(-m\log p)\bigr), \\
w_\infty(f) &= \Bigl(\frac{\Gamma'}{\Gamma}(1/4) - \log \pi\Bigr) f(0) - \int_{0}^\infty \frac{f(t) + f(-t) - 2f(0)}{1-e^{-2t}} e^{-t/2}\, dt,
\end{align*}
and where
\[
\hat f(\xi) = \int_{-\infty}^\infty e^{i\xi t} f(t)\, dt,
\]
to the function $\varphi_{x,c,\eps}$.

Let $\Delta = \varphi_{x,c,\eps} - f_x$ and assume $x>2/\abs{\log(\eps)}$. It then suffices to prove the following identities:
\begin{eqnarray}
&w_s(\hat \varphi_{x,c,\eps}) &=\sum_{\rho} a_{c,\eps}(\rho) \frac{x^\rho-1}{\rho} - x - \log x + 1 \label{e:ws_phi}\\
&w_f(\varphi_{x,c,\eps}) &= -\psi_{c,\eps}(x) \label{e:wf_phi}\\
&w_\infty(f_x) &= - \log x - \frac \gamma 2 - \half \log \pi - \half \log(1-x^{-2})\label{e:w8_f} \\
&w_\infty(\Delta) &= \Theta(8 \eps\abs{\log\eps})\label{e:w8_d}.
\end{eqnarray}

The identities \eqref{e:ws_phi} and \eqref{e:wf_phi} follow directly from the definitions of the functionals. So we begin with the proof of \eqref{e:w8_f}. We have
\begin{equation*}
w_\infty(f_x) = \half\frac{\Gamma'}{\Gamma}(1/4) - \half \log \pi - \int_0^{\log x} \frac{1-e^{-t/2}}{1-e^{-2t}}\, dt + \int_{\log x}^\infty \frac{e^{-t/2}}{1-e^{-2t}}\, dt.
\end{equation*}
Using
\begin{equation*}
\half\frac{\Gamma'}{\Gamma}(1/4) = \int_0^\infty \frac{e^{-2t}}{t}-\frac{e^{-t/2}}{1-e^{-2t}}\, dt
\end{equation*}
and
\begin{equation*}
- \int_0^{\log x} \frac{1-e^{-t/2}}{1-e^{-2t}}\, dt = - \log x + \int_0^{\log x} \frac{e^{-t/2}-e^{-2t}}{1-e^{-2t}}\,dt,
\end{equation*}
we get
\begin{align*}
w_\infty(f_x) &= -\log x - \half \log \pi + \int_0^{\log x} \frac{e^{-2t}}{2t} -  \frac{e^{-2t}}{1-e^{-2t}}\,dt + \int_{\log x}^\infty \frac{e^{-2t}}{2t}\, dt\\
&= -\log x - \half \log \pi - \half \log(1-x^{-2}) + \half \lim_{\delta\searrow 0} \bigl(E_1(2\delta) - \log(1-e^{-2\delta})\bigr),
\end{align*}
where
\[
E_1(y) = \int_{y}^\infty \frac{e^{-t}}{t}\,dt
\]
denotes the first exponential integral. Since
\begin{equation*}
 E_1(y) = -\gamma - \log(y) + O(y)
\end{equation*}
holds for $y\searrow 0$ \cite[p. 40]{Olver97}, we get
\begin{equation*}
\lim_{\delta\searrow 0} \bigl(E_1(2\delta) - \log(1-e^{-2\delta})\bigr) = -\gamma +\log\Bigl(\frac{1-e^{-2\delta}}{2\delta}\Bigr) = - \gamma,
\end{equation*}
which concludes the proof of \eqref{e:w8_f}.

It remains to show \eqref{e:w8_d}, and we start by bounding $\Delta(t)$:

\begin{lemma}
Let $\eps$ and $x$ be as in the proposition. Then $\Delta(t)$ vanishes for $t\notin B_\eps(0)\cup B_\eps(\log x)$. Moreover, we have
\begin{equation}\label{e:dtmt}
\Delta(t) + \Delta(-t) = 2\Delta(0) + \Theta(2t)\quad\quad\text{for $0\leq t\leq \eps$,}
\end{equation}
\begin{equation}\label{e:dbex}
\abs{\Delta(t)} \leq \frac{1}{2}e^\eps \sqrt{x}\quad\quad\text{for $t\in B_\eps(\log x)$,}
\end{equation}
and
\begin{equation}\label{e:d0}
\abs{\Delta(0) } \leq \eps.
\end{equation}
\end{lemma}
\begin{proof}
Under the conditions imposed on $x$ and $\eps$, we have $B_\eps(0)\cap B_\eps(\log x) = \emptyset$ and
\begin{equation}\label{e:exp-approx}
e^{t+\tau} = e^t + \Theta(2\abs\tau)
\end{equation}
for $\max\{\abs t,\abs \tau\} \leq \eps$.

Since $\exp(\cdot /2)\ast \eta_{c,\eps}(t) = \lambda_{c,\eps} \exp(t /2)$ this gives
\[
\Delta(0) = \frac{1}{2\lambda_{c,\eps}}\int_0^\eps \eta_{c,\eps}(\tau)\bigl( e^{\tau/2} - e^{-\tau/2}\bigr)\, d\tau = \Theta(\eps),
\]
so we get \eqref{e:d0}. Moreover, we have
\begin{align*}
\Delta(t) + \Delta(-t) &= \frac{1}{\lambda_{c,\eps}} \int_t^\eps \eta_{c,\eps}(\tau) \bigl(e^{\frac{\tau - t}{2}} - e^{\frac{t-\tau}{2}}\bigr)\, dt\\
&= \frac{1}{\lambda_{c,\eps}} \int_t^\eps \eta_{c,\eps}(\tau) \bigl(e^{\tau/2} - e^{-\tau/2}\bigr)\, dt + \Theta(t) \\
&= \frac{1}{\lambda_{c,\eps}} \int_0^\eps \eta_{c,\eps}(\tau) \bigl(e^{\tau/2} - e^{-\tau/2}\bigr)\, dt + \Theta(2t), \\
\end{align*}
which gives \eqref{e:dtmt}. The remaining inequality \eqref{e:dbex} follows easily from
\[
\Delta(t) = \frac{\chi_{(\log x,\infty)}(t)}{\lambda_{c,\eps}} \int_{t-\log x}^\eps \eta_{c,\eps}(\tau) e^{\frac{t-\tau}{2}}\, d\tau - \frac{\chi_{(0,\log x)(t)}}{\lambda_{c,\eps}} \int_{-\eps}^{t-\log x} \eta_{c,\eps}(\tau) e^{\frac{t-\tau}{2}}\, d\tau,
\]
which holds for $t\in B_\eps(\log x)$.
\end{proof}

Now, we divide the integral in $w_\infty(\Delta)$ as follows
\begin{multline}\label{e:w8Delta-Integral}
\int_{0}^\infty \frac{\Delta(t)+\Delta(-t)-2\Delta(0)}{1-e^{-2t}}e^{-t/2}\, dt = \int_{0}^\eps \frac{\Delta(t)+\Delta(-t)-2\Delta(0)}{1-e^{-2t}}e^{-t/2}\, dt \\ - 2 \int_\eps^\infty  \frac{\Delta(0)}{1-e^{-2t}}e^{-t/2}\, dt + \int_{B_\eps(\log x)} \frac{\Delta(t)}{1-e^{-2t}}e^{-t/2}\, dt.
\end{multline}
Since the mapping $t\mapsto  \frac{1-\exp(-2t)}{t}$ is monotonously decreasing in $(0,\infty)$, we have
\begin{equation}\label{e:denom-lower}
1-e^{-2t} \geq 1.8\, t
\end{equation}
for $0 \leq t \leq \eps \leq 0.1.$ So, using \eqref{e:dtmt}, we obtain the bound
\begin{equation*}
\int_{0}^\eps \frac{\abs{\Delta(t)+\Delta(-t)-2\Delta(0)}}{1-e^{-2t}}e^{-t/2}\, dt \leq \int_{0}^\eps \frac{2t}{1.8\, t}\, dt \leq 1.2\,\eps
\end{equation*}
for the first integral on the right hand side of \eqref{e:w8Delta-Integral}.

For the second integral we use \eqref{e:d0} and the bound $\abs{\log\eps} \geq 2.3$, which gives
\begin{equation*}
2 \abs{\Delta(0)} \int_{\eps}^\infty \frac{e^{-t/2}}{1-e^{-t}}\,dt \leq 2 \abs{\Delta(0)} \, \abs{\log\frac{e^{\eps/2}-1}{e^{\eps/2}+1}} 
\leq 2\,\eps\, \abs{\log\frac{\eps}{2\cdot 2.1}} \leq 3.4\,\eps\, \abs{\log\eps}.
\end{equation*}

It remains to bound the third integral on the right hand side of \eqref{e:w8Delta-Integral}. From \eqref{e:denom-lower} we get
\begin{equation*}
1-e^{-2t} \geq 1 - \exp(2\eps-2\log x) \geq 1 - \exp\Bigl(-\frac{4}{\abs{\log\eps}}\Bigr) \geq \frac{2}{\abs{\log\eps}}
\end{equation*}
for $t\in B_\eps(\log x)$ which, together with \eqref{e:dbex}, implies
\begin{equation*}
\int_{B_\eps(\log x)} \frac{\abs{\Delta(t)}}{1-e^{-2t}}e^{-t/2}\, dt \leq \half e^{\eps/2}\sqrt{x}\frac{e^{\eps/2}}{\sqrt{x}} \int_{B_\eps(\log x)} \frac{dt}{1-e^{-2t}} \leq  \eps\, \abs{\log\eps}.
\end{equation*}

By the Gau\ss -Digamma theorem  \cite[Theorem 1.2.7]{Andrews1999}, we have
\begin{equation*}
\frac{\Gamma'}{\Gamma}(1/4) = -\gamma - \frac{\pi}{2} -3\log 2,
\end{equation*}
so \eqref{e:d0} gives the bound
\begin{equation*}
\abs{\Bigl(\frac{\Gamma'}{\Gamma}(1/4) - \log \pi\Bigr) \Delta(0)} \leq 5.4\,\eps
\end{equation*}
for the remaining summand in $w_\infty(\Delta)$. Therefore, we arrive at
\begin{equation*}
\abs{w_\infty(\Delta)} \leq \eps(5.4+1.2 +(3.4+1)\abs{\log\eps}) \leq 8\,\eps\,\abs{\log\eps},
\end{equation*}
which concludes the proof of the proposition.
\end{proof}

%% file: zsum.tex
\section{Bounding the sum over zeros}
We provide several bounds for parts of the sum over zeros in the explicit formula for $\psi_{c,\eps}(x).$ First we truncate the sum, making use of the sharp cuttoff property of the Logan function.

\begin{prop}\label{s:psi-zsum-remainder}
Let $x>1$, $\eps\leq 10^{-3}$ and $c\geq 3$. Then we have
\begin{equation}\label{e:rem1}
\sum_{\abs{\Im(\rho)} > \frac c\eps} \abs{a_{c,\eps}(\rho) \frac{x^{\rho}}{\rho}} \leq 0.16 \frac{x+1}{\sinh(c)} e^{0.71\sqrt{c\eps}} \log(3c) \log\Bigl(\frac c\eps\Bigr).
\end{equation}

Furthermore, if $a\in(0,1)$ such that $a\frac{c}{\eps}\geq 10^{3}$ holds, and if the Riemann hypothesis holds for all zeros with imaginary part in $(0,\frac c\eps]$, then we have
\begin{equation}\label{e:rem2}
\sum_{\frac{ac}\eps <  \abs{\Im(\rho)} \leq \frac c\eps} \abs{a_{c,\eps}(\rho) \frac{x^{\rho}}{\rho}} \leq\frac{1+11c\eps}{\pi c a^2} \log\Bigl(\frac c\eps\Bigr)\frac{\cosh(c\sqrt{1-a^2})}{\sinh(c)} \sqrt{x}.
\end{equation}
\end{prop}

\begin{proof}
Since $\exp(t/2)$ is convex and $\eta_{c,\eps}$ is non-negative and even, we have 
\[
 \lambda_{c,\eps}\exp(t/2) = \exp(\cdot/2)\ast\eta_{c,\eps}(t) \geq  \exp(t/2),
\]
and therefore $\lambda_{c,\eps}\geq 1$. Thus
\begin{equation*}
\abs{a_{c,\eps}(\rho)\frac{x^\rho}{\rho}} \leq x^{\Re(\rho)}\frac{\abs{\ell_{c,\eps}\bigl(\frac\rho i - \frac 1{2i}\bigr)} }{\abs{\Im(\rho)}}
\end{equation*}
holds for every non-trivial zero $\rho$. From this one obtains \eqref{e:rem2} from \cite[Lemma 5.5]{BuetheA}, pairing $\rho$ and $1-\overline\rho$ for every zero off the critical line, and \eqref{e:rem1} follows from the following lemma.
\end{proof}
\begin{lemma}\label{l:logan-zsum}
Let $0<\eps < 10^{-3}$ and let $c\geq 3$. Then we have
\begin{equation*}
\sum_{\abs{\Im(\rho)}>\frac c\eps} \frac{\abs{\ell_{c,\eps}\bigl(\frac{\rho}{i}- \frac{1}{2i}\bigr)}}{\abs{\Im(\rho)}} \leq 0.32 \frac{e^{0.71 \sqrt{c\eps}}}{\sinh(c)} \log(3c) \log\Bigl(\frac c\eps\Bigr).
\end{equation*}
\end{lemma}
\begin{proof}
This is a more flexible version of \cite[Lemma 2.4]{FKBJ}, which is proven in detail in  \cite{BuetheDiss}. We give a brief outline of the proof: We may weaken the condition $T>10^6$ to $T\geq 100$ by replacing the constant $0.4$ by $0.82$ in Corollary 2.2 and by replacing $M+6$ by $M+18$ in Corollary 2.3. In the proof of Lemma 2.4 we replace the definition of $f(z)$ by $\frac{\sinh(c)}{c} e^{-0.71\sqrt{c\eps}} \, \ell_{c,1}(z)$. It is then straightforward to show that (2.7) and (2.8) and the final inequality remain true, which gives the desired result.
\end{proof}

For the remaining part of the zeros, we will also be needing the following lemma.

\begin{lemma}\label{l:zero-reci}
Let $t_2>t_1\geq 14$. Then we have
\begin{equation}\label{e:zsum-part0}
\sum_{t_1\leq \Im(\rho) < t_2} \frac{1}{\Im(\rho)} \leq 
\frac{1}{4\pi}\Bigl[ \log\Bigl(\frac {t_2}{2\pi}\Bigr)^2 - \log\Bigl(\frac{t_1}{2\pi}\Bigr)^2\Bigr] + \Theta\Bigl(5\frac{\log t_1}{t_1}\Bigr),
\end{equation}
and for $t_2\geq 5000$ we have
\begin{equation*}
\sum_{0<\Im(\rho) < t_2} \frac{1}{\Im(\rho)} \leq 
\frac{1}{4\pi} \log\Bigl(\frac {t_2}{2\pi}\Bigr)^2.
\end{equation*}
\end{lemma}

\begin{proof}
Let $N(t)$ denote the zero-counting function. Using the notation $N(t) = g(t) + R(t)$, where $g(t) = \frac{t}{2\pi}\log \frac{t}{2\pi e} + \frac 78$, we get

\begin{equation*}
\sum_{t_1\leq \Im(\rho) < t_2} \frac{1}{\Im(\rho)} = \int_{t_1}^{t_2} \frac{g'(t)}{t}\, dt
+ \int_{t_1}^{t_2} \frac{dR(t)}{t}.
\end{equation*}

Here the first integral gives the main term in \eqref{e:zsum-part0}. Furthermore, Rosser's estimate \cite[p. 223]{Rosser41} implies $\abs{R(t)} \leq \log t$ for $t\geq 14$. Consequently, we get
\begin{align*}
\int_{t_1}^{t_2} \frac{dR(t)}{t} 
&= \Bigl[\frac{R(t)}{t}\Bigr]_{t_1}^{t_2} + \int_{t_1}^{t_2} \frac{R(t)}{t^2}\, dt \\
&\leq 2\frac{\log t_1}{t_1} + \int_{t_1}^{t_2} \frac{\log t}{t^2}\, dt \\
&\leq 4 \frac{\log t_1}{t_1} + \frac{1}{t_1} \leq 5 \frac{\log t_1}{t_1}.
\end{align*}

In particular, we have
\begin{align*}
\sum_{0<\Im(\rho) < t_2} \frac{1}{\Im(\rho)} 
&\leq   \frac{1}{4\pi} \log\Bigl(\frac {t_2}{2\pi}\Bigr)^2 + \sum_{0<\Im(\rho) < 5000} \frac{1}{\Im(\rho)} - \frac{1}{4\pi} \log\Bigl(\frac{5000}{2\pi}\Bigr)^2 + 5\frac{\log (5000)}{5000} \\
&\leq \frac{1}{4\pi} \log\Bigl(\frac {t_2}{2\pi}\Bigr)^2 + 3.54 - 3.55 + 0.0086 
<\frac{1}{4\pi} \log\Bigl(\frac {t_2}{2\pi}\Bigr)^2 
\end{align*}
for $t_1\geq 5000$.
\end{proof}

%% file: psum.tex
\section{Bounding the sum over prime powers}
The modified Chebyshov function $\psi_{c,\eps}$ can be used to trivially bound $\psi(x)$, choosing $\alpha=1$ in Proposition \ref{p:psi-psice}, but one obtains considerably better results choosing $\alpha$ close to zero and bounding the sum over prime powers.

We introduce the auxiliary functions
\[
\mu_{c,\eps}(t) =
\begin{cases}
-\int_{-\infty}^t\nce(\tau)\,d\tau	& t < 0,\\
-\mu_{c,\eps}(-t)			&t>0,\\
 0					&t=0
\end{cases}
\]
and
\[
\nu_{c,\eps}(t) = \int_{-\infty}^t \mu_{c,\eps}(\tau)\, d\tau.
\]

 \begin{prop}\label{s:Mxceps-bound}
Let $0\leq \alpha < 1$, $x>100$, and $\eps<10^{-2}$, such that
\[
B = \frac{\eps x e^{-\eps} \abs{\nu_c(\alpha)}}{2(\mu_c)_+(\alpha)} > 1
\]
holds. We define
\[
 A(x,c,\eps,\alpha) = e^{2\eps} \log (e^{\eps}x) \Bigl[ \frac{2 \eps\, x \,\abs{\nu_c(\alpha)}}{\log B} + 2.01 \eps\sqrt{x} +\frac 12 \log\log(2 x^2) \Bigr].
\]
Then we have
\[
\psi(e^{-\alpha\eps}x) \leq \psi_{c,\eps}(x) + A(x,c,\eps,\alpha)
\]
and
\[
\psi(e^{\alpha\eps}x) \geq \psi_{c,\eps}(x) - A(x,c,\eps,\alpha).
\]
\end{prop}

We will use the following two Lemmas from \cite{BuetheA}.

\begin{lemma}[{\cite[Lemma 3.5]{BuetheA}}]\label{l:ppow-bound}
Let $x\geq 100$, $\eps\leq \frac 1{100}$ and let $I=[e^{-\eps}x,e^\eps x]$. Then we have
\begin{equation*}
\sum_{\substack{p^m\in I\\ m\geq 2}} \frac{1}{m} \leq 4.01 \eps \sqrt x + \log\log(2x^2).
\end{equation*}
\end{lemma}

\begin{lemma}[{\cite[Lemma 5.8]{BuetheA}}]\label{l:psum-mu}
Let $x>1,\eps<1$ and $\alpha\in(0,1)$, such that
\[
B := \frac{\eps x e^{-\eps}\abs{\nu_c(\alpha)}}{2\mu_c(\alpha)} > 1
\]
holds. Furthermore, let $I^+_\alpha = [e^{\alpha\eps} x, e^\eps x]$ and $I^-_\alpha = [e^{-\eps} x, e^{-\eps\alpha} x]$. Then we have
\begin{equation*}
\sum_{ p \in I_\alpha^\pm} \abs{\mu_{c,\eps}\left(\log \frac p x\right)} \leq 2 \frac{\eps x e^\eps \abs{\nu_c(\alpha)}}{\log B}.
\end{equation*}
\end{lemma}

\begin{proof}[Proof of Proposition \ref{s:Mxceps-bound}]
By Proposition \ref{p:psi-psice}, it suffices to show that
\begin{equation*}
\abs{ \sum_{p^m\in I^\pm_\alpha} \frac 1m M_{x,c,\eps} (p^m) } \leq A(x,c,\eps,\alpha).
\end{equation*}

From \eqref{e:exp-approx} and \eqref{e:M-def} one easily obtains the bound
\begin{equation*}
M_{x,c,\eps}(t) = \frac{\log t}{\lce} \mu_{c,\eps}\Bigl(\log \frac{t}{x}\Bigr)(1+\Theta(\eps)) \leq \frac {e^\eps}2 \log(e^\eps x).
\end{equation*}

Then Lemma \ref{l:psum-mu} gives the bound
\[
2 e^{2\eps} \log(e^\eps x) \frac{\eps x \abs{\nu_c(\alpha)}}{\log B}
\]
for the contribution of the prime numbers in $I^\pm_\alpha$, and Lemma \ref{l:ppow-bound} gives the bound
\[
\frac{e^{\eps}} 2 \log(e^\eps x)\bigl[ 4.01 \eps \sqrt{x} + \log\log (2x)^2\bigr]
\]
for the contribution of the remaining prime powers in $I^\pm_\alpha$.

\end{proof}

Analyzing the asymptotic behavior of $\mu_c(\alpha)$ and $\nu_c(\alpha)$ as functions of $c$ for arbitrary $\alpha$ seems difficult. However, we can do this for the case $\alpha=0$, which is usually not too far from the optimal choice. To this end, we introduce the modified Bessel function of the first kind for real parameters $\gamma\geq 0$ by

\begin{equation}\label{e:def-bessel-i}
I_\gamma(x) = \Bigl(\frac{x}{2}\Bigr)^\gamma \sum_{n=0}^\infty \frac{(x/2)^{2n}}{n! \Gamma(\gamma + n+ 1)}.
\end{equation}
Then we have the following proposition.

\begin{prop}\label{p:nu0-asymp}
For $c_0>0$ let
\[
D(c_0) = \sqrt{\frac {\pi c_0} 2} \frac{I_1(c_0)}{\sinh(c_0)}.
\]
Then the inequalities
\[
\frac{D(c_0)}{\sqrt{2\pi c}} \leq \abs{\nu_c(0)} \leq \frac{1}{\sqrt{2\pi c}}
\]
hold for all $c\geq c_0$. Furthermore, we have $D(c_0)\nearrow 1$ for $c_0\to \infty$.
\end{prop}

\begin{proof}
Since
\[
\abs{\nu_c(0)} = \frac{I_1(c)}{2\sinh(c)}
\]
\cite[p. 15]{FKBJ} and since $I_{1/2}(x) = \sqrt{\frac{2}{\pi x}} \sinh(x)$ the assertion follows directly from the following lemma.
\end{proof}

\begin{lemma}\label{l:BesselQuot-Monotonie}
Let $\alpha,\beta\in[0,\infty)$ such that $\alpha<\beta$ holds. Then the function
\[
\frac{I_\beta(x)}{I_\alpha(x)}
\]
is positive and monotonously increasing in $(0,\infty)$ and converges to $1$ for $x\to \infty$.
\end{lemma}

\begin{proof}
The proof is based on the Sturm Monotony Principle \cite{Sturm1833}, \cite[p. 518]{Watson66}. We define the auxiliary function
\[
f_\gamma(x) = \sqrt{x} I_\gamma(x).
\]
The Bessel differential equation
\[
\frac{d^2}{dx^2}I_\gamma + \frac{1}{x}\frac{d}{dx}I_\gamma - \Bigl(1 + \frac{\gamma^2}{x^2}\Bigr)I_\gamma = 0
\]
then implies
\[
\frac{d^2}{dx^2}f_\gamma - \Bigl(1 - \frac{1}{4x} + \frac{\gamma^2}{x^2}\Bigr) f_\gamma = 0.
\]
Consequently, we have
\[
f_\beta f''_\alpha - f''_\beta f_\alpha = \frac{\beta^2-\alpha^2}{x^2} f_\alpha f_\beta > 0
\]
in $(0,\infty)$ and thus
\[
\Bigl[ f_\beta f'_\alpha - f'_\beta f_\alpha\Bigr]_{\eps}^x > 0
\]
for $x>\eps$ and every $\eps>0$. Since
\[
f_\beta f'_\alpha - f'_\beta f_\alpha = I_\beta\Bigl( xI'_\alpha + I_\alpha\Bigr) - I_\alpha\Bigl(x I'_\beta + I_\beta\Bigr) 
\]
vanishes for $x \to 0$ we thus get
\[
f_\beta f'_\alpha - f'_\beta f_\alpha \geq 0.
\]
Consequently, the function $f_\beta/f_\alpha = I_\beta/I_\alpha$ increases monotonously in $(0,\infty)$, and since
\[
I_\gamma(x) \sim \frac{e^x}{\sqrt{2\pi x}}
\]
holds for every $\gamma\geq 0$, it converges to 1 for $x\to\infty$.
\end{proof}

%% file: chebb.tex
\section{Bounds of Chebyshov type}
The previous results give rise to a simple method to calculate bounds of the form
\[
\abs{\psi(x)-x} \leq \delta_0 x\quad\quad\text{for $x\geq x_0$},
\]
which will be needed in the proof of the main result. 

\begin{thm}\label{t:ChebB}
Let $0<\eps<10^{-3}$, $c\geq 3$, $x_0\geq 100$ and $\alpha\in[0,1)$ such that the inequality
\[
B_0 := \frac{\eps e^{-\eps}x_0 \abs{\nu_c(\alpha)}}{2(\mu_c)_+(\alpha)} > 1
\]
holds. We denote the zeros of the Riemann zeta function by $\rho = \beta + i \gamma$ with $\beta,\gamma\in\R$. Then, if $\beta=1/2$ holds for $0<\gamma\leq c/\eps$, the inequality
\[
\abs{\psi(x)-x} \leq x \cdot e^{\alpha \eps}  \bigl(\CE_1 + \CE_2 + \CE_3\bigr)
\]
holds for all $x\geq e^{\alpha\eps}x_0$, where
\begin{align}
\CE_1 &=  e^{2\eps}\log(e^\eps x_0)\Bigl[\frac{2\eps\abs{\nu_c(\alpha)} }{\log B_0} + \frac{2.01\eps}{\sqrt{x_0}} + \frac{\log\log(2x_0^2)}{2 x_0}\Bigr] + (e^{\alpha \eps} - 1),  \notag \\
\CE_2 &=0.16\frac{1+x_0^{-1}}{\sinh(c)}e^{0.71\sqrt{c\eps}}\log\Bigl(\frac{c}{\eps}\Bigr),\notag \\
\intertext{and}
\CE_3 &= \frac{2}{\sqrt{x_0}} \sum_{0<\gamma \leq c/\eps}\frac{\ell_{c,\eps}(\gamma)}{\gamma} + \frac{2}{x_0}.\label{e:E3-def}
\end{align}
\end{thm}

It is noteworthy that Theorem 1 gives better estimates than the more sophisticated method in \cite{FK15} in a large range, as can be seen from Tables \ref{tb:platt} and \ref{tb:gourdon}.

\begin{proof}
Under the conditions of the theorem we get
\begin{equation*}
\psi(e^{-\alpha\eps} x) - e^{-\alpha\eps} x \leq \psi_{c,\eps}(x) - e^{-\alpha\eps} x + \frac{A(x_0,c,\eps,\alpha)}{x_0} x \leq \psi_{c,\eps}(x) - x + \CE_1 x
\end{equation*}
from Proposition \ref{s:Mxceps-bound}, since $A(x,c,\eps,\alpha)/x$ decreases monotonously. A similar calculation for the lower bound then gives
\begin{equation*}
\abs{\psi(e^{\pm\alpha\eps} x) - e^{\pm\alpha\eps} x} \leq \abs{\psi_{c,\eps}(x) - x} + \CE_1 x.
\end{equation*}
Furthermore, we get
\begin{equation*}
\abs{\psi_{c,\eps}(x) - x} \leq\sum_{\abs{\Im(\rho)} \leq c/\eps}  \abs{a_{c,\eps}(\rho)\frac{x^\rho}\rho} + 2  +  \CE_2 x  \leq (\CE_2 + \CE_3) x
\end{equation*}
from Propositions \ref{p:exp-form} and \ref{s:psi-zsum-remainder}, so the assertion follows.
\end{proof}

\renewcommand{\arraystretch}{1.5}
\begin{centering}
\begin{table}
\caption{Bounds for $T\leq 3.061\times 10^{10}$. The value $\delta_0$ is an upper bound for $e^{\alpha\eps}(\CE_1 + \CE_2 + \CE_3)$ in Theorem \ref{t:ChebB}, applied with $\eps=c/T$.}\label{tb:platt}
\begin{tabular}{c|c|c|c|c}
$e^{\alpha\eps}x_0$	&	$c$ & $T$ & $\alpha$	&	$\delta_0$ \\
\hline
$e^{45}$		&$25$		&$3.5\times 10^{9}$  			&$0.11$		&$1.11742\times 10^{-8}$		\\
$e^{50}$		&$30$		&$3.061\times 10^{10}$		&$0.11$		&$1.16465\times 10^{-9}$		\\
$e^{55}$		&$30$		&$3.061\times 10^{10}$		&$0.1$			&$2.88434\times 10^{-10}$		\\
$e^{60}$		&$28$		&$3.061\times 10^{10}$		&$0.09$		&$2.08162\times 10^{-10}$		\\
$e^{65}$		&$28$		&$3.061\times 10^{10}$		&$0.09$		&$1.96865\times 10^{-10}$		\\
$e^{70}$		&$28$		&$3.061\times 10^{10}$		&$0.08$		&$1.91910\times 10^{-10}$		\\
$e^{80}$		&$28$		&$3.061\times 10^{10}$		&$0.07$		&$1.84848\times 10^{-10}$		\\
$e^{90}$		&$29$		&$3.061\times 10^{10}$		&$0.06$		&$1.79330\times 10^{-10}$		\\
$e^{100}$		&$29$		&$3.061\times 10^{10}$		&$0.05$		&$1.75185\times 10^{-10}$		\\
$e^{500}$		&$29$		&$3.061\times 10^{10}$		&$0.01$		&$1.47067\times 10^{-10}$		\\
$e^{1000}$	&$29$		&$3.061\times 10^{10}$		&$0.005$		&$1.43770\times 10^{-10}$		\\
$e^{3000}$	&$29$		&$3.061\times 10^{10}$		&	$0.001$	&$1.41594\times 10^{-10}$			\\
\end{tabular}
\end{table}
\end{centering}
\begin{centering}
\begin{table}
\caption{Bounds for $T\leq 2.445\times 10^{12}$. The value $\delta_0$ is an upper bound for $e^{\alpha\eps}(\CE_1 + \CE_2 + \CE_3)$ in Theorem \ref{t:ChebB}, applied with $\eps=c/T$.}\label{tb:gourdon}
\begin{tabular}{c|c|c|c|c}
$e^{\alpha\eps}x_0$	&	$c$ & $T$ & $\alpha$	&	$\delta_0$ \\
\hline

$e^{55}$		&$39$	&$8.5\times 10^{11}$			&$0.1$		&$1.12494 \times 10^{-10}$	\\
$e^{60}$		&$33$	&$2.445\times 10^{12}$		&$0.11$		&$1.22147 \times 10^{-11}$		\\
$e^{65}$		&$33$	&$2.445\times 10^{12}$		&$0.1$		&$3.57125 \times 10^{-12}$		\\
$e^{70}$		&$33$	&$2.445\times 10^{12}$		&$0.09$		&$2.79233 \times 10^{-12}$		\\
$e^{75}$		&$32$	&$2.445\times 10^{12}$		&$0.08$		&$2.70358 \times 10^{-12}$		\\
$e^{80}$		&$33$	&$2.445\times 10^{12}$		&$0.08$		&$2.61079 \times 10^{-12}$		\\
$e^{90}$		&$33$	&$2.445\times 10^{12}$		&$0.07$		&$2.52129 \times 10^{-12}$		\\
$e^{100}$	&$33$	&$2.445\times 10^{12}$		&$0.06$		&$2.45229 \times 10^{-12}$		\\
$e^{500}$	&$33$	&$2.445\times 10^{12}$		&$0.012$		&$1.99986 \times 10^{-12}$		\\
$e^{1000}$	&$33$	&$2.445\times 10^{12}$		&$0.005$		&$1.94751 \times 10^{-12}$		\\
$e^{2000}$	&$33$	&$2.445\times 10^{12}$		&$0.003$		&$1.92155 \times 10^{-12}$		\\
$e^{3000}$	&$33$	&$2.445\times 10^{12}$		&$0.001$		&$1.91298 \times 10^{-12}$	\\
$e^{4000}$	&$33$	&$2.445\times 10^{12}$		&$0.001$		&$1.90866 \times 10^{-12}$		\\
\end{tabular}
\end{table}
\end{centering}

\subsection{Numerical estimates for $\CE_1$ and $\CE_3$}

The sum over zeros in \eqref{e:E3-def} can either be evaluated, which is recommended if $c/\eps$ is small, or the sum can be estimated piecewise, using the following lemma.

\begin{lemma}  Let $c,\eps>0$ and let $14\leq T_0 < T_1 < c/\eps$. Then we have
\begin{equation*}
\sum_{T_0 \leq \gamma < T_1} \frac{\ell_{c,\eps}(\gamma)}{\gamma}
 \leq \frac{\ell_{c,\eps}(T_0)}{4\pi} \left[\log\Bigl(\frac{T_1}{2\pi}\Bigr)^2 - \log\Bigl(\frac{T_0}{2\pi}\Bigr)^2 + 20\pi\frac{\log(T_1)}{T_1}\right].
\end{equation*}
\end{lemma}
\begin{proof}
This follows directly from $\ell_{c,\eps}$ being monotonously decreasing in $[0,c/\eps]$ and Lemma \ref{l:zero-reci}.
\end{proof}

The values $\mu_c(\alpha)$ and $\nu_c(\alpha)$ can be evaluated by power series representations, as shown in \cite{FKBJ}. Alternatively, these values can be bounded by Riemann sums.

\begin{lemma} \label{l:riemann-sums} Let $\alpha\in(0,1)$, $K\in\N$ and let $h=\frac{1-\alpha}{K}$. Then we have
\begin{equation*}
{hc}\sum_{k=0}^{K-1} \frac{I_0(c\sqrt{2kh-k^2h^2})} {2\sinh(c)}\leq \mu_c(\alpha) \leq hc\sum_{k=1}^{K} \frac{I_0(c\sqrt{2kh-k^2h^2})}{2\sinh(c)}
\end{equation*}
and
\begin{multline*}
h^2c\sum_{k=0}^{K-1} \sum_{j=0}^k \frac{I_0(c\sqrt{2jh-j^2h^2})}{2\sinh(c)} \leq \abs{\nu_c(\alpha)} \leq h^2c\sum_{k=1}^{K} \sum_{j=1}^k \frac{I_0(c\sqrt{2jh-j^2h^2})}{2\sinh(c)}.
\end{multline*}
\end{lemma}
\begin{proof}
This follows from $\mu_c' = -\eta_{c,1}$ in $(0,1)$ and $\nu_c'=\mu_c$, since both $\eta_{c,1}$ and $\mu_c$ are monotonously decreasing and non-negative in this region.
\end{proof}

%% file: ppnt.tex
\section{A partial prime number theorem}
We now come to the main result of this paper, the proof of Schoenfeld's bounds \cite{Schoenfeld76} for the functions $\psi(x)$,
\begin{equation*}
\pi(x) = \sum_p \chi^*_{[0,x]} (p),\quad \vartheta(x)=\sum_p \chi^*_{[0,x]}(p)\log(p),\quad\text{and}\quad \pi^*(x) = \sum_{p^m} \frac 1m \chi^*_{[0,x]}(p^m)  ,
\end{equation*}
in limited ranges under partial RH assumptions. This is a slight improvement of \cite[Theorem 6.1]{BuetheDiss}.

\begin{thm}\label{t:partial-prime}
Let $T>0$ such that the Riemann hypothesis holds for $0<\Im(\rho)\leq T$. Then, under the condition $4.92 \sqrt{\frac{x}{\log x}} \leq T$, the following estimates hold:
\begin{align}
\abs{\psi(x) - x} &\leq \frac{\sqrt{x}}{8\pi}\log(x)^2 &\text{for $x>59$,} \notag \\
\abs{\vartheta(x) - x} &\leq \frac{\sqrt{x}}{8\pi}\log(x)^2 &\text{for $x>599$,} \notag \\
\abs{\pi^*(x) - \li(x)} &\leq \frac{\sqrt{x}}{8\pi}\log(x) &\text{for $x>59$,}\label{e:pp-pi-star} \\
\intertext{and}
\abs{\pi(x) - \li(x)} &\leq \frac{\sqrt{x}}{8\pi}\log(x) &\text{for $x>2657$.}\label{e:pp-pi}
\end{align}
\end{thm}

In particular the numerical verification in \cite{Platt15} ($T\approx 3.061\times 10^{10}$) gives these bounds for $x\leq 1.89\times 10^{21}$, the result in \cite{FKBJ} ($T=10^{11}$) gives them for $x\leq 2.1\times 10^{22}$ and the result in \cite{gourdon04} ($T\approx 2.445\times 10^{12}$) gives them for $x\leq 1.4\times 10^{25}$.

\begin{proof}
We will first prove the stronger bounds
\begin{equation}\label{e:strong-psi-bound}
\abs{\psi(x)-x} \leq \frac{\sqrt x}{8\pi}\log(x)\bigl(\log(x) - 3\bigr)\quad\quad\text{for $x\geq 5000$,}
\end{equation}
and
\begin{equation}\label{e:strong-theta-bound}
\abs{\vartheta(x)-x} \leq \frac{\sqrt x}{8\pi}\log(x)\bigl(\log(x) - 2\bigr)\quad\quad\text{for $x\geq 5000$}.
\end{equation}
These imply the bounds in \eqref{e:pp-pi-star} and \eqref{e:pp-pi} for $x\geq 5000$, since if $(f,g)$ is one of the tuples $(\psi,\pi^*)$ or $(\vartheta,\pi)$, we have
\[
g(x)-g(a) = \li(x)-\li(a) - \frac{x-f(x)}{\log(x)} + \frac{a-f(a)}{\log a} - \int_{a}^x \frac{t-f(t)}{t\log(t)^2}\, dt
\]
by partial summation, and so we get
\begin{multline*}
\abs{\pi^*(x) - \li(x)} \leq \frac{\sqrt x}{8\pi}(\log(x) -3) + \abs{\pi^*(5000)-\li(5000) - \frac{\psi(5000)-5000}{\log(5000)}} \\
 + \frac{\sqrt{x}}{4\pi} - \frac{\sqrt{5000}}{4\pi} < \frac{\sqrt x}{8\pi} \log(x)
\end{multline*}
and
\begin{multline*}
\abs{\pi(x) - \li(x)} \leq \frac{\sqrt x}{8\pi}(\log(x) -2) + \abs{\pi(5000)-\li(5000) - \frac{\vartheta(5000)-5000}{\log(5000)}} \\
 + \frac{\sqrt{x}}{4\pi} - \frac{\sqrt{5000}}{4\pi} < \frac{\sqrt x}{8\pi} \log(x).
\end{multline*}
For the remaining values of $x$ the validity of the claimed inequalities is easily checked by a short computer calculation (the author did this with the pari/gp calculator).

We will prove \eqref{e:strong-psi-bound} for $x\geq 10^{19}$ first, choosing
\begin{align*}
c &= \frac 12 \log(x) + 5 \\
\intertext{and}
\eps &= \frac{\log(x)^{3/2}}{8 \sqrt{x}}
\end{align*}
in Proposition \ref{p:exp-form}. In particular, we then have $c> 26$ and $\eps < 1.2\times 10^{-8}$. If we take into account that
\[
\abs{\sum_{\rho}\frac{a_{c,\eps}(\rho)}{\rho}} = \zabs{\sum_{\Im(\rho)>0}\frac{a_{c,\eps}(\rho)}{\rho(1-\rho)}} \leq \frac{e^{\eps/2}\abs{1+i100}	}{100} \zsum_\rho \frac{1}{\rho} \leq 0.024
\]
holds under these conditions, \eqref{e:exp-form} can be simplified to
\begin{equation}\label{e:pp-start}
x-\psi_{c,\eps}(x) = \zsum_{\rho} \frac{a_{c,\eps}(\rho)}{\rho} x^{\rho} + \Theta(2).
\end{equation}
Furthermore, we have
\[
\frac{c}{\eps} \leq 4.92 \sqrt{\frac x{\log x}} \leq T,
\]
so we may assume $\Re(\rho) = 1/2$ for all zeros $\rho$ with imaginary part up to $c/\eps$.

We divide the sum in \eqref{e:pp-start} into three parts. For $\abs{\Im(\rho)}>c/\eps$ we get 
\begin{align}
\sum_{\abs{\Im(\rho)} > \frac c\eps} \abs{a_{c,\eps}(\rho) \frac{x^{\rho}}{\rho}} 
	&\leq 0.16 \frac{x+1}{\sinh(c)} e^{0.71\sqrt{c\eps}} \log(3c) \log\Bigl(\frac c\eps\Bigr) \notag\\
	&\leq 0.0013\sqrt{x}\log(x)\log\log(x) =: \mathcal{E}_1(x)\label{e:zsum1-bound}
\end{align}
from Proposition \ref{s:psi-zsum-remainder}. Furthermore, choosing $a=\sqrt{\frac 2c}$ in Proposition \ref{s:psi-zsum-remainder} gives
\begin{align}
\sum_{\frac{\sqrt{2c}}\eps < \abs{\Im(\rho)} \leq \frac c\eps} \abs{a_{c,\eps}(\rho) \frac{x^{\rho}}{\rho}}
	&\leq \frac{1+11c\eps}{2\pi} \log\Bigl(\frac c\eps\Bigr)\frac{\cosh(c\sqrt{1-a^2})}{\sinh(c)} \sqrt{x}\notag\\
	& \leq \frac{1.001}{4\pi e}\log(x)\sqrt{x} \leq 0.03\log(x)\sqrt{x} =:\mathcal{E}_2(x).\label{e:zsum2-bound}
\end{align}

For the remaining part of the sum we bound $\abs{a_{c,\eps}(\rho)/\rho}$ trivially by $1/\abs{\Im(\rho)}$ and use Lemma \ref{l:zero-reci}, which gives 
\begin{align}
\sum_{ 0< \abs{\Im(\rho)} \leq \frac{\sqrt{2c}}\eps} \abs{a_{c,\eps}(\rho) \frac{x^{\rho}}{\rho}}
	&\leq \frac{\sqrt x}{2\pi}\log\left(\frac{\sqrt{2c}}{2\pi\eps}\right)^2 \label{e:zsum3-bound}\\
	&\leq \frac{\sqrt x}{2\pi}\left(\frac 12 \log(x) + \log(1.45) - \log\log(x)\right)^2 \notag\\
	&\leq \frac{\sqrt{x}}{8\pi}\log(x)^2 + \mathcal{E}_3(x), \notag
\end{align}
where
\begin{multline*}
 \mathcal{E}_3(x) = \sqrt{x}\Bigl(0.061\log(x) + 0.16\log\log(x)^2 \\ + 0.024 - 0.15\log(x)\log\log(x) - 0.114\log\log(x)\Bigr).
\end{multline*}

Next, we treat the difference $\psi(x) - \psi_{c,\eps}(x)$. Lemma \ref{l:BesselQuot-Monotonie} implies
\begin{equation*}
\frac{0.98}{\sqrt{2\pi c}} \leq \abs{\nu_c(0)} = \frac{I_1(c)}{2\sinh(c)} \leq \frac{1}{\sqrt{2\pi c}}
\end{equation*}
 for $c>26$, so that we get
\begin{equation}\label{e:psum-bound}
\begin{split}
\abs{\psi(x)-\psi_{c,\eps}(x)} &\leq \frac{2.001 \sqrt{x}\log(x)^{5/2}}
{8\sqrt{\pi(\log(x)+10)} } \log\Bigl( \frac{0.97 \sqrt x\log(x)^{3/2}}{8 \sqrt{\pi(\log(x)+10)} } \Bigr)^{-1}\\
 &\quad\quad+ \frac{2.02}{8}\log(x)^{5/2} + 0.51 \log\log(2x^2)\log(x) 
 \end{split}
\end{equation}
from Proposition \ref{s:Mxceps-bound}. Since we have $\sqrt{\frac{\log(x)}{\log(x)+10}} \geq  0.9$, the first summand on the right hand side is bounded by
\begin{equation*}
\mathcal{E}_4(x) := 0.283 \sqrt x\frac{ \log(x)^{3/2}}{\sqrt{\log(x) + 10}}.
\end{equation*}
So if we define
\[
\mathcal{E}_5(x) := 0.26\log(x)^{5/2} + 0.51\log(x)\log\log(2x)^2 + 2,
\]
we get
\begin{equation*}
\abs{\psi(x)-x} \leq \frac{\sqrt{x}}{8\pi}\log(x)^2 + \CE_1(x) + \CE_2(x) + \CE_3(x) + \CE_4(x) + \CE_5(x)
\end{equation*}
from \eqref{e:pp-start},  \eqref{e:zsum1-bound},  \eqref{e:zsum2-bound},  \eqref{e:zsum3-bound}, and  \eqref{e:psum-bound}. Differentiating with respect to the variable $y=\log(x)$ shows that
\[
\frac{1}{\sqrt{x}\log(x)}\bigl(\CE_1(x) + \CE_2(x) + \CE_3(x) + \CE_4(x) + \CE_5(x)\bigr)
\]
is monotonously decreasing for $x\geq 10^{19}$ and smaller than $-\frac{3}{8\pi}$, so \eqref{e:strong-psi-bound} holds in this region.

For $\exp(18) \leq x \leq \exp(44)$  \eqref{e:strong-psi-bound} can be proven by calculating a sufficient amount of Chebyshov bounds with the method from the previous section. To this end, it suffices to verify
\begin{equation}\label{e:cheb-xn}
\abs{\psi(x)-x} \leq \delta_n x
\end{equation}
for $x\geq y_n = \exp(n/4)$, with a $\delta_n$ satisfying
\begin{equation}\label{e:cheb-cond}
\delta_n y_n \leq e^{-1/8} \frac{\sqrt{y_n}}{8 \pi}\log(y_n)(\log(y_n)-3),
\end{equation}
since then \eqref{e:cheb-xn} implies \eqref{e:strong-psi-bound} for $x\in[y_n,y_{n+1}]$ by concavity of the right hand side. This has been carried out with the choice $x_0 = \exp(-\alpha\eps)y_n$, $c = n/8 + 5$, $T = 2\sqrt{y_n}$, $\eps=c/T$ and $\alpha=0.2$ in Theorem \ref{t:ChebB} for $72\leq n \leq 129$, and with the altered choice $T=4\sqrt{y_n/\log(y_n)}$ and $\alpha=0.1$ for $129\leq n\leq 175$. In all cases \eqref{e:cheb-cond} turned out to hold.

For the remaining $x\in[5000,\exp(18)]$ the validity of \eqref{e:strong-psi-bound} is easily checked numerically by evaluating $\psi(x)$ at all prime powers in this interval.

Since we have
\[
\psi(x)-\psi(\sqrt x) \leq \vartheta(x)\leq \psi(x),
\]
\eqref{e:strong-psi-bound} implies \eqref{e:strong-theta-bound} for $x\geq 10^{11}$. For the remaining $x$  \eqref{e:strong-theta-bound} follows from the bound
\[
0 \leq x - \vartheta(x) \leq 1.938 \sqrt{x}\quad\quad\quad\text{for $5000\leq x\leq 10^{11}$},
\]
which the author obtained numerically.
\end{proof}